\documentclass[12pt]{amsart}

\usepackage{amsmath,amssymb,amsthm}
\usepackage[all]{xy}
\usepackage[dvipdfm,colorlinks=true]{hyperref}

\numberwithin{equation}{section}

\SelectTips{eu}{12}

\newtheorem{theorem}{Theorem}[section]
\newtheorem{proposition}[theorem]{Proposition}
\newtheorem{lemma}[theorem]{Lemma}
\newtheorem{corollary}[theorem]{Corollary}

\theoremstyle{definition}
\newtheorem{definition}[theorem]{Definition}
\newtheorem{example}[theorem]{Example}

\theoremstyle{remark}
\newtheorem{remark}[theorem]{Remark}

\newcommand{\R}{\mathbb{R}}
\newcommand{\Z}{\mathcal{Z}}

\title{Decompositions of suspensions of spaces involving polyhedral products}

\author{Kouyemon Iriye}
\address{Department of Mathematics and Information Sciences, Osaka Prefecture University, Sakai, 599-8531, Japan}
\email{kiriye@mi.s.osakafu-u.ac.jp}
\author{Daisuke Kishimoto}
\address{Department of Mathematics, Kyoto University, Kyoto, 606-8502, Japan}
\email{kishi@math.kyoto-u.ac.jp}

\thanks{K.I. is supported by JSPS KAKENHI (No. 26400094), and D.K. is supported by JSPS KAKENHI (No. 25400087)}

\date{\today}
\subjclass[2010]{}
\keywords{}

\begin{document}

\maketitle

\begin{abstract}
Two homotopy decompositions of supensions of spaces involving polyhedral products are given. The first decomposition is motivated by the decomposition of suspensions of polyhedral products in \cite{BBCG}, and is a generalization of the retractile argument of James \cite{J}. The second decomposition is on the union of an arrangement of subspaces called diagonal subspaces, and generalizes the result in \cite{La}.
\end{abstract}

% baseline
\baselineskip 16pt

\section{Introduction}

A space which is now called a polyhedral product is constructed from a collection of pairs of spaces in accordance with the combinatorial information of a given abstract simplicial complex, where the collection is labeled by vertices of the given simplicial complex. By definition polyhedral products are related to fundamental objects in combinatorics, geometry, and topology such as Stanley-Reisner rings and their derived algebras, graph products of groups (e.g. right-angled Artin and Coxeter groups), quasitoric manifolds, coordinate subspace arrangements, higher order Whitehead products, and so on. The aim of this paper is to provide two kinds of homotopy decompositions of suspensions of spaces involving polyhedral products: one is a generalization of the decompositions of \cite{BBCG} and \cite{ABBCG}, and the other is a decomposition of the union of arrangements of special subspaces called diagonal subspaces which include polyhedral products as subspaces. We briefly explain the backgrounds of these decompositions.

An important property of polyhedral products is the existence of retractions onto certain ``sub''polyhedral products, where this kind of retraction property also appears in other contexts \cite{AC,ACG,ACTG,KT}. By using this retraction property, Bahri, Bendersky, Cohen, and Gitler \cite{BBCG} gave a decomposition of suspensions of polyhedral products, and we aim at generalizing this decomposition. It is actually obtained by the retractile argument due to James \cite{J} which provides a decomposition of suspensions of spaces satisfying a certain retraction property, and we will generalize the retractile argument which is the first decomposition. Our decomposition has a naturality which cannot be obtained by the retractile argument, and recovers, of course, a decomposition of suspensions of polyhedral products by Bahri, Bendersky, Cohen, and Gitler \cite{BBCG} and also the decomposition of suspensions of simplicial spaces by Adem, Bahri, Bendersky, Cohen, and Gitler \cite{ABBCG}. We here note that in \cite{ABBCG} it is pointed out that the decomposition of suspensions of polyhedral products can be obtained from the decomposition of suspensions of simplicial spaces, but polyhedral products do not seem to fit to the context of simplicial spaces. 

The second space which we decompose is the union of an arrangement of special subspaces called diagonal subspaces which includes important subspace arrangements such as braid arrangements, where we abbreviate this union as the diagonal arrangement. The decomposition of a suspension of diagonal arrangements was formerly studied by Labassi \cite{La} in a special case, and Sadok Kallel posed a question whether the result of Labassi can be generalized to general diagonal arrangements under a certain dimensional condition imposed on the special case of Labassi. We give an affirmative answer to this question which is our second decomposition. These diagonal arrangements include special polyhedral products as subspaces, though in general we cannot describe properties of these polyhedral products as subspaces of the diagonal arrangements, i.e. properties of the inclusion. But under a certain dimensional condition, we can describe properties of the inclusion which enable us to prove the decomposition.

The authors are grateful to the referees for useful advises and helpful comments, where they pointed out that it is sufficient to assume retractibility of $\Sigma X$ instead of $X$ in Theorem \ref{first-decomp} and showed a generalization mentioned in Remark \ref{generalization}.

\section{Retractile spaces over posets}

In this section we consider a space over a poset with natural retractions, and prove a decomposition of a suspension of its certain colimit. To explain what we are going to do, we start with a simple example. Consider the diagram
$$\xymatrix{X\ar[r]&X\times Y\\
\ast\ar[u]\ar[r]&Y\ar[u]}$$
of spaces. Then we see that every arrow has a retraction, and it induces a filtration 
$$*\subset X\vee Y\subset X\times Y.$$
By the above retractions, the filtration splits after a suspension to yield the decomposition
$$\Sigma(X\times Y)\simeq \Sigma(X\vee Y)\vee\Sigma(X\wedge Y)$$
which is natural with respect to $X$ and $Y$. The aim of this section is to generalize this situation. Let $P$ be a poset. We regard $P$ as a category by pointing upward, that is, for $p,q\in P$, $p\to q$ in the category means $p\le q$ in the poset. We assume two conditions on $P$:

\begin{enumerate}
\item $P$ is graded, i.e. $P=\coprod_{n\ge 0}P^n$ as sets and for $p\in P^n$ and $q\in P^m$, $p<q$ implies $n<m$.
\item $P$ is a lower semilattice, i.e. any $p,q\in P$ have the greatest lower bound $p\wedge q$.
\end{enumerate}

Let $X$ be a space over $P$ which is a functor from $P$ to the category of pointed topological spaces. A map between spaces over $P$ is a natural transformation as usual. The grading of $P$ defines a filtration
$$X^{\underline{0}}\subset X^{\underline{1}}\subset\cdots\subset X^{\underline{n}}\subset X^{\underline{n+1}}\subset\cdots,$$
where $X^{\underline{n}}=\mathrm{colim}\,X\vert_{P^\le n}$ for the restriction $X\vert_{P^\le n}$ of $X$ to the subcategory $P^{\le n}:=\coprod_{0\le k\le n}P^k$. We say that $X$ is $n$-cofibrant if the canonical map $X^{\underline{i}}\to X^{\underline{i+1}}$ is a cofibration for $i=0,\ldots,n-1$. There is a sufficient condition for the $n$-cofibrancy.

\begin{lemma}
[cf. \cite{Li}]
\label{cofibrant}
If all arrows of $X\vert_{P^{\le n}}$ are cofibrations, $X$ is $n$-cofibrant.
\end{lemma}

We now define natural retractions in the diagram $X$, and state the main result of this section.

\begin{definition}
We say that $X$ is retractile if every arrow $\iota_{q,p}\colon X_p\to X_q$ admits a retraction $\rho_{p,q}$ satisfying
$$\rho_{p,r}\circ\iota_{r,q}=\rho_{p,q}\quad\text{and}\quad\rho_{p,r}=\rho_{p,q}\circ\rho_{q,r}\quad\text{for}\quad p<q<r.$$
\end{definition}

Let $X,Y$ be retractile spaces over $P$. We say that a map $f\colon X\to Y$ of spaces over $P$ preserves retractions if $\rho^Y_{p,q}\circ f_q=f_p\circ\rho_{p,q}^X$ for any $p<q\in P$, where $\rho_{p,q}^X$ and $\rho_{p,q}^Y$ are the retractions of $X$ and $Y$, respectively. Put $X(p):=X_p/\mathrm{colim}\,X\vert_{P_{<p}}$ for $p\in P$, where $P_{<p}:=\{q\in P\,\vert\,q<p\}$.

\begin{theorem}
\label{first-decomp}
Let $X$ be a space over a graded lower semilattice $P$. If $X$ is $n$-cofibrant and $\Sigma X$ is retractile, then there is a homotopy equivalence
$$\Sigma X^{\underline{n}}\simeq\Sigma\bigvee_{p\in P^{\le n}}X(p)$$
which is natural with respect to maps of spaces over $P$ preserving retractions.
\end{theorem}

\begin{remark}
\label{generalization}
We can generalize Theorem \ref{first-decomp} by weakening the condition to that there are maps $\bar{\rho}_{p,q}\colon\Sigma X_q\to\Sigma X(p)$ for any $q>p\in P$ such that the composite $\Sigma X_p\xrightarrow{\Sigma\iota_{q,p}}\Sigma X_q\xrightarrow{\bar{\rho}_{p,q}}\Sigma X(p)$ is the quotient map and $\bar{\rho}_{p,q}\circ\Sigma\iota_{q,r}=\bar{\rho}_{p,r}$ for $p<r<q\in P$, where $\iota_{p,q}$ is an arrow in $X$. Indeed, we can construct a quotient map $\bar{\rho}_p^m\colon X^{\underline{m}}\to X(p)$ for $p\in P^k$ with $k\le m$ satisfying a property analogous to Lemma \ref{retraction-colim}, so the proof of Theorem \ref{first-decomp} works for this situation.
\end{remark}

The rest of this section is devoted to the proof of this theorem, and we prepare two lemmas.

\begin{lemma}
\label{retraction-colim}
For $p\in P^k$ with $k\le m$, there is a retraction $\rho_p^m\colon X^{\underline{m}}\to X_p$ of the canonical map $X_p\to X^{\underline{m}}$ satisfying 
$$\rho_p^m\circ i=\rho_p^\ell$$
for $k\le\ell\le m$ and the canonical map $i\colon X^{\underline{\ell}}\to X^{\underline{m}}$.
\end{lemma}

\begin{proof}
Let $\iota_{r,q}\colon X_q\to X_r$ be the arrow in $X$ for $q<r\in P$. Fix $p\in P^k$. Since $P$ is a lower semilattice, we can define a space $Y$ over $P$ by putting $Y_q=X_{p\wedge q}$ and the arrow $Y_q\to Y_r$ to be $\iota_{p\wedge r,p\wedge q}$.  Then the map $\theta_q:=\iota_{q,p\wedge q}\colon Y_q=X_{p\wedge q}\to X_q$ defines a map $\theta\colon Y\to X$ of spaces over $P$. Indeed for $q<r$, we have 
$$\theta_r\circ\iota_{p\wedge r,p\wedge q}=\iota_{r,p\wedge r}\circ\iota_{p\wedge r,p\wedge q}=\iota_{r,p\wedge q}=\iota_{r,q}\circ\iota_{q,p\wedge q}=\iota_{r,q}\circ\theta_q.$$
The map $\tau_q:=\rho_{p\wedge q,q}\colon X_q\to X_{p\wedge q}=Y_q$ also defines a map $\tau\colon X\to Y$ of spaces over $P$ since for $q<r$, we have
$$\tau_r\circ\iota_{r,q}=\rho_{p\wedge r,r}\circ\iota_{r,q}=\iota_{p\wedge r,p\wedge q}\circ\rho_{p\wedge q,q}=\iota_{p\wedge r,p\wedge q}\circ\tau_q.$$
By definition, we have $\tau\circ\theta=1_Y$ and $Y^{\underline{n}}=X_p$ for $n\ge k$. Thus the induced map $X^{\underline{m}}\to Y^{\underline{m}}=X_p$ from $\tau$ is the desired retraction.
\end{proof}

\begin{lemma}
[{\cite[Theorem 4.2]{HMR}}]
\label{split}
If there is a homotopy retraction $r$ of the suspension $\Sigma f$ of a cofibration $f\colon A\to B$, then the map
$$r+\Sigma\pi\colon \Sigma B\to\Sigma A\vee\Sigma(B/A)$$
is a homotopy equivalence, where $\pi\colon B\to B/A$ is the projection.
\end{lemma}

\begin{proof}[Proof of Theorem \ref{first-decomp}]
We show that the map
$$\sum_{p\in P^{\le n}}\Sigma\pi_p\circ\rho^n_p\colon \Sigma X^{\underline{n}}\to\Sigma\bigvee_{p\in P^{\le n}}X(p)$$
is a homotopy equivalence which implies the desired naturality, where $\rho_p^n$ is as in Lemma \ref{retraction-colim} for $\Sigma X$ and $\pi_p\colon X_p\to X(p)$ is the projection. Let $\epsilon^n$ denote the map in the statement. We induct on $n$. For $n=0$, the theorem is trivial. Suppose that $\epsilon^{n-1}$ is a homotopy equivalence. Since the restriction $\epsilon^n\vert_{\Sigma X^{\underline{n-1}}}$ is homotopic to $\epsilon^{n-1}$ by Lemma \ref{retraction-colim}, the map 
$$(\epsilon^{n-1})^{-1}\circ\sum_{p\in P^{\le n-1}}\Sigma\pi_p\circ\rho^n_p\colon \Sigma X^{\underline{n}}\to\Sigma X^{\underline{n-1}}$$ 
is a left homotopy inverse of the canonical map $\Sigma X^{\underline{n-1}}\to\Sigma X^{\underline{n}}$. Then it follows from Lemma \ref{split} that the map
$$\bar{\pi}+\sum_{p\in P^{\le n-1}}\Sigma\pi_p\circ\rho^n_p\colon \Sigma X^{\underline{n}}\to\Sigma(X^{\underline{n}}/X^{\underline{n-1}})\vee\Sigma\bigvee_{p\in P^{\le n-1}}X(p)$$ 
is a homotopy equivalence, where $\bar{\pi}\colon \Sigma X^{\underline{n}}\to\Sigma(X^{\underline{n}}/X^{\underline{n-1}})$ is the projection. It is obvious that $\Sigma(X^{\underline{n}}/X^{\underline{n-1}})=\bigvee_{p\in P^n}X(p)$ and the projection $\bar{\pi}$ is homotopic to $\sum_{p\in P^n}\Sigma\pi_p\circ\rho^n_p$, completing the proof.
\end{proof}

\section{Applications of Theorem \ref{first-decomp}}

This section shows three applications of Theorem \ref{first-decomp} which recover the results of \cite{BBCG} and \cite{ABBCG}.

\subsection{Product spaces}

We consider the product space $X_1\times\cdots\times X_m$. Let $[m]$ denote a finite set $\{1,\ldots,m\}$. We define a space $X$ over a lattice $2^{[m]}$, the power set of $[m]$, by
$$X_I:=\prod_{i\in I}X_i$$
for $I\subset[m]$. Then it is obvious that $X$ is retractile. By definition, we have $X^{\underline{n}}$ is the generalized fat wedge
$$\{(x_1,\ldots,x_m)\in X_1\times\cdots\times X_m\,\vert\,\text{at least }m-n\text{ of }x_i\text{'s are basepoints}\}.$$
and $X(I)=\bigwedge_{i\in I}X_i$ for $I\subset[m]$. Then by Theorem \ref{first-decomp} we get the standard decomposition
$$\Sigma(X_1\times\cdots\times X_m)\simeq\Sigma\bigvee_{I\subset[m]}\bigwedge_{i\in I}X_i.$$
The case $m=2$ is the above mentioned decomposition of a product of two spaces. This decomposition of product spaces is generalized to that of polyhedral products as below.

\subsection{Polyhedral products}

Let $K$ be an abstract simplicial complex on the vertex set $[m]$, and let $(\underline{X},\underline{A}):=\{(X_i,A_i)\}_{i\in[m]}$ be a collection of pairs of pointed spaces indexed by the vertex set of $K$. The polyhedral product $\Z_K(\underline{X},\underline{A})$ is defined by
$$\Z_K(\underline{X},\underline{A}):=\bigcup_{\sigma\in K}(\underline{X},\underline{A})^\sigma\quad(\subset X_1\times\cdots\times X_m),$$
where $(\underline{X},\underline{A})^\sigma=Y_1\times\cdots\times Y_m$ for $Y_i=X_i$ and $A_i$ according as $i\in\sigma$ and $i\not\in\sigma$. Polyhedral products are connected with several areas of mathematics as mentioned in Section 1, and this connection is actually made through homotopy invariants in many cases. So it is particularly important to describe the homotopy types of polyhedral products. In studying the homotopy types of polyhedral products, the decomposition of suspensions of polyhedral products due to Bahri, Bendersky, Cohen, and Gitler \cite{BBCG} is fundamental as in \cite{GT,IK1,IK2}, and we here recover this decomposition from Theorem \ref{first-decomp}. For $I\subset[m]$, put $K_I:=\{\sigma\subset I\,\vert\,\sigma\in K\}$ and $(\underline{X}_I,\underline{A}_I):=\{(X_i,A_i)\}_{i\in I}$. Then we get a polyhedral product $\Z_{K_I}(\underline{X}_I,\underline{A}_I)$ for which there is the inclusion
$$\iota_{J,I}\colon\Z_{K_I}(\underline{X}_I,\underline{A}_I)\to\Z_{K_J}(\underline{X}_J,\underline{A}_J)$$
for $I\subset J\subset[m]$ by using the basepoints, where we assume $\Z_{K_\emptyset}(\underline{X}_\emptyset,\underline{A}_\emptyset)$ is a point. For $I\subset J\subset[m]$, the projection $\prod_{j\in J}X_j\to\prod_{i\in I}X_i$ induces a map 
$$\rho_{I,J}\colon\Z_{K_J}(\underline{X}_J,\underline{A}_J)\to\Z_{K_I}(\underline{X}_I,\underline{A}_I)$$
which is a retraction of the inclusion $\iota_{J,I}$. This retraction obviously satisfies the following property.

\begin{lemma}
\label{retractile-Z}
For $I,J\subset[m]$, we have
$$\rho_{I\cap J,I}\circ\rho_{I,I\cup J}=\rho_{I\cap J,J}\circ\rho_{J,I\cup J}.$$
\end{lemma}

The assignment
$$I\mapsto\Z_{K_I}(\underline{X}_I,\underline{A}_I)$$
defines a space over a lattice $2^{[m]}$ which we denote by $Z$. We define the grading of $2^{[m]}$ by the cardinality of subsets. Then the associated filtration
$$*=Z^{\underline{0}}\subset Z^{\underline{1}}\subset\cdots\subset Z^{\underline{m}}=\Z_K(\underline{X},\underline{A})$$
is the fat wedge filtration which plays the fundamental role in describing the homotopy type of the special polyhedral product $\Z_K(C\underline{X},\underline{X})$ as in \cite{IK2}. We can define a space $\widehat{\Z}_K(\underline{X},\underline{A})$ by replacing the direct product with the smash product in the definition of the polyhedral product $\Z_K(\underline{X},\underline{A})$ above. Then for $I\subset[m]$, we have
$$Z(I)=\widehat{\Z}_{K_I}(\underline{X}_I,\underline{A}_I).$$
Note that by Lemma \ref{cofibrant}, if each $(X_i,A_i)$ is an NDR pair, then $Z$ is $m$-cofibrant. By Lemma \ref{retractile-Z}, $Z$ is also retractile, so by Theorem \ref{first-decomp} we obtain:

\begin{theorem}
[Bahri, Bendersky, Cohen, and Gitler \cite{BBCG}]
\label{BBCG}
If $(\underline{X},\underline{A})$ is a collection of NDR pairs, there is a homotopy equivalence 
$$\Sigma\Z_K(\underline{X},\underline{A})\simeq\Sigma\bigvee_{\emptyset\ne I\subset[m]}\widehat{\Z}_{K_I}(\underline{X}_I,\underline{A}_I)$$
which is natural with respect to $(\underline{X},\underline{A})$.
\end{theorem}

\begin{example}
\label{DJ}
Let $(\underline{X},\underline{*})$ denote $n$-copies of a pair of a space and its basepoint $(X,*)$. Note that $\widehat{Z}_{K_I}(\underline{X},\underline{*})$ is either a point or $\widehat{X}^{|I|}$ according as $I\not\in K$ and $I\in K$, where $\widehat{X}^n$ denotes the smash product of $n$-copies of $X$. Then by Theorem \ref{BBCG} we have
$$\Sigma\Z_K(\underline{X},\underline{*})\simeq\Sigma\bigvee_{\sigma\in K}\widehat{X}^{|\sigma|}$$
which is natural with respect to $X$, where this will be used below.
\end{example}

\subsection{Simplicial spaces}

Recall that a simplicial space $X$ is a sequence of spaces $X_0,X_1,\ldots$ equipped with the face maps $d_0,\ldots,d_n\colon X_n\to X_{n-1}$ and the degeneracy maps $s_0,\ldots,s_n\colon X_n\to X_{n+1}$ for all $n$ which satisfy the well known simplicial identity. We construct a space $\underline{X}$ over a graded lattice $2^{[n]}$ for fixed $n$ from a simplicial space $X$, where the grading of the lattice $2^{[n]}$ is given by the cardinality of subsets as above. For $I\subset[n]$, we put
$$\underline{X}_I:=X_{|I|}.$$
For $i\not\in I$, we put $\iota_{I\cup i,I}\colon\underline{X}_I\to\underline{X}_{I\cup i}$ to be the degeneracy map $s_j$, where $I\cup i=\{i_1<\cdots<i_{|I|+1}\}$ and $i_{j-1}=i$. Then we easily see that this generates a space $\underline{X}$ over $2^{[n]}$. Moreover, by the simplicial identity $d_j$ is a retraction of $s_j$ which makes $\underline{X}$ retractile also by the simplicial identity. We next describe $\underline{X}^{\underline{m}}$ in terms of the degeneracy maps. We set
$$S^k(X_n):=\{x\in X_n\,\vert\,x=s_{i_1}\cdots s_{i_k}(y)\text{ for some }y\in X_{n-k}\text{ and }i_1>\cdots>i_k\}$$
for $k\ge 0$ and $S^{-1}(X_n)$ to be a point. By the simplicial identity $d_is_i=1$, the map $s_i\colon X_m\to s_i(X_m)$ is a homeomorphism, so we have 
$$\underline{X}^{\underline{n-k}}\cong S^k(X_n).$$
Then we get
$$S^k(X_n)/S^{k+1}(X_n)\cong\underline{X}^{\underline{n-k}}/\underline{X}^{\underline{n-k-1}}=\bigvee_{I\subset[n],\,|I|=n-k}\underline{X}(I)$$
which is observed in \cite{ABBCG}. Thus we obtain:

\begin{theorem}
\label{ABBCG}
Let $\underline{X}$ be a space over $2^{[n]}$ associated with a simplicial space $X$. If $\underline{X}$ is $n$-cofibrant, then
$$\Sigma\underline{X}^{\underline{n}}\simeq\Sigma\bigvee_{k=0}^nS^k(X_n)/S^{k+1}(X_n)\quad\text{and}\quad S^k(X_n)/S^{k+1}(X_n)\cong\bigvee_{I\subset[n],\,|I|=n-k}\underline{X}(I)$$
which are natural with respect to simplicial maps.
\end{theorem}

\begin{example}
Regard $[n]$ as a discrete space, and consider the standard cosimplicial structure on $\{[n]\}_{n\ge 1}$, where the indexing differs from the usual case by one. For a space $X$, we define a simplicial space $\underline{X}$ by
$$\underline{X}_{n-1}:=\mathrm{map}([n],X)\quad(=X^n).$$
Then we have $X^{\underline{n-1}}=X^n$ and $S^k(\underline{X}_{n-1})$ is
$$\{(x_1,\ldots,x_n)\in X^n\,\vert\,x_{i_1}=x_{i_1+1},\ldots,x_{i_k}=x_{i_{k+1}}\text{ for some }1\le i_1<\cdots<i_k\le n-1\}$$
which is the union of a special diagonal arrangement investigated below. Thus Theorem \ref{ABBCG} gives a decomposition of $\Sigma(X^n)$ which is not the standard one in Subsection 3.1. This type of construction applies to the spaces of commuting elements in a Lie group as in \cite{ABBCG}.
\end{example}

\section{Diagonal arrangements}

Homotopy decompositions are fundamental powerful tools in studying topology of subspace arrangements and their complements. Here are two examples: Ziegler and \u{Z}ivaljevi\'c \cite{ZZ} decompose the one point compactification of affine subspace arrangements, from which one can deduce the well known Goresky-MacPherson formula \cite{GM} on the (co)homology of the complements of affine subspace arrangements, and Bahri, Bendersky, Cohen, and Gitler \cite{BBCG} decompose suspensions of polyhedral products including coordinate subspace arrangements and their complements, from which one can deduce Hochster's formula on related Stanley-Reisner rings, whereas Grbi\'c and Theriault \cite{GT} and the authors \cite{IK1,IK2} study the desuspension of the decomposition of $\Sigma\Z_K(C\underline{X},\underline{X})$, where $(C\underline{X},\underline{X})$ is the sequence of cones and their bases. In this section we consider a decomposition of the union of an arrangement of the following special subspaces. Fix a space $X$. For a subset $\sigma\subset[m]$, the subspace of $X^m$ defined by
$$\Delta_\sigma(X):=\{(x_1,\ldots,x_m)\in X^m\,\vert\,x_{i_1}=\cdots=x_{i_k}\text{ for }\{i_1,\ldots,i_k\}=[m]-\sigma\}$$
is called the diagonal subspace of $X^m$ associated with $\sigma$. The arrangement of diagonal subspaces
$$\Delta_{\sigma_1}(X),\ldots,\Delta_{\sigma_k}(X)\quad\text{for}\quad\sigma_1,\ldots,\sigma_k\subset[m]$$
is called the diagonal arrangement, where it is sometimes called the hypergraph arrangement since it is determined by the hypergraph whose vertex set is $[m]$ and edges are $\sigma_1,\ldots,\sigma_k$. One can regard diagonal arrangements as a generalization of the braid arrangement which corresponds to the diagonal arrangement defined by all subsets of $[m]$ with cardinality $m-2$. Topology and combinatorics of diagonal arrangements have been studied in several directions. See \cite{Ko,PRW,Ki,KS,La,MW,M} for example.  In this paper, we are interested in the topology of the union $\Delta_{\sigma_1}(X)\cup\cdots\cup\Delta_{\sigma_k}(X)$.

We set convention and notation on diagonal arrangements. By removing the inessential part, we may assume that $\sigma_1\cup\cdots\cup\sigma_k=[m]$ for the above diagonal arrangement, and it is useful to consider all diagonal subspaces included in $\Delta_{\sigma_1}(X),\ldots,\Delta_{\sigma_k}(X)$, for example, to express the union as a colimit, that is, we consider all diagonal subspaces $\Delta_\sigma(X)$ for $\sigma\in K$, where $K$ is a simplicial complex generated by $\sigma_1,\ldots,\sigma_k$. Then we assume that all diagonal arrangements have the form
$$\{\Delta_\sigma(X)\,\vert\,\sigma\in K\}$$
for a simplicial complex $K$ on the vertex set $[m]$. For example, the braid arangement is the case when $K$ is the $(m-3)$-skeleton of the $(m-1)$-dimensional full simplex. We put
$$\Delta_K(X):=\bigcup_{\sigma\in K}\Delta_\sigma(X).$$
Observe that the polyhedral product $\Z_K(\underline{X},\underline{*})$ is a subspace of $\Delta_K(X)$, where $(\underline{X},\underline{*})$ denotes $m$-copies of $(X,*)$. 

Labassi \cite{La} shows that the suspension $\Sigma\Delta_K(X)$ decomposes into a wedge of smash products of copies of $X$ when $K$ is the $(m-d-1)$-skeleton of the $(m-1)$-simplex and $2d>m$, in which case $\Delta_K(X)$ consists of all $(x_1,\ldots,x_m)\in X^m$ such that at least $d$-tuple of $x_i$'s are identical. The proof for this decomposition in \cite{La} heavily depends on the symmetry of the skeleta of simplices, and so it cannot apply to general $K$. However, Sadok Kallel poses the following problem to the authors: is there a homotopy decomposition of $\Sigma\Delta_K(X)$ for $ 2(\dim K+1)<m$ which includes Labassi's decomposition? We give an affirmative answer to this question as:

\begin{theorem}
\label{second-decomp}
If $X$ has the homotopy type of a connected CW-complex and $2(\dim K+1)<m$, then
$$\Sigma\Delta_K(X)\simeq\Sigma(\bigvee_{\sigma\in K}\widehat{X}^{|\sigma|}\vee\widehat{X}^{|\sigma|+1})$$
where $\widehat{X}^k$ is the smash product of $k$-copies of $X$ for $k>0$ and $\widehat{X}^0$ is a point.
\end{theorem}

As a corollary, we  calculate the Euler characteristic of the complement of the diagonal arrangement $\mathcal{M}_K(X)=X^m-\Delta_K(X)$.

\begin{corollary}
\label{euler-char}
Let $X$ be a closed connected $n$-manifold. If $2(\dim K+1)<m$, the Euler characteristic of $\mathcal{M}_K(X)$ is given by
$$\chi(\mathcal{M}_K(X))=\chi(X)^m-(-1)^{mn}\chi(X)(1+\sum_{\emptyset\ne\sigma\in K}(\chi(X)-1)^{|\sigma|}).$$
\end{corollary}

\begin{proof}
Since $X$ is a compact manifold, $\Delta_K(X)$ is a compact, locally contractible subset of an $mn$-manifold $X^m$. Then by the Poincar\'e-Alexander duality \cite[Proposition 3.46]{H}, there is an isomorphism 
$$H_i(X^m,\mathcal{M}_K(X);\mathbb{Z}/2)\cong H^{mn-i}(\Delta_K(X);\mathbb{Z}/2),$$
implying that $\chi(X^m,\mathcal{M}_K(X))=(-1)^{mn}\chi(\Delta_K(X))$.
Thus since $\chi(\widehat{X}^k)=(\chi(X)-1)^k+1$ for $k\ge 1$, it follows from Theorem \ref{second-decomp} that
$$\chi(X^m,\mathcal{M}_K(X))=(-1)^{mn}\chi(X)(1+\sum_{\emptyset\ne\sigma\in K}(\chi(X)-1)^{|\sigma|}).$$
Therefore the proof is completed by the equality $\chi(X^m)=\chi(X^m,\mathcal{M}_K(X))+\chi(\mathcal{M}_K(X))$.
\end{proof}

\begin{remark}
Corollary \ref{euler-char} does not hold without compactness of $X$. For example, if $X=\R$ (hence $n=1$) and $K$ is the $(m-3)$-skeleton of the full $(m-1)$-simplex, $\mathcal{M}_K(X)$ is homotopy equivalent to $m!$ points, implying $\chi(\mathcal{M}_K(X))=m!$ which differs from Corollary \ref{euler-char}.
\end{remark}

\section{Proof of Theorem \ref{second-decomp}}

The outline of the proof of Theorem \ref{second-decomp} is as follows. As mentioned above, the polyhedral product $\Z_K(\underline{X},\underline{*})$ is a subspace of $\Delta_K(X)$. In general the inclusion $\Z_K(\underline{X},\underline{*})\to\Delta_K(X)$ is not a fiber inclusion of a homotopy fibration unlike our case so that we cannot connect properties of polyhedral products to $\Delta_K(X)$. But under the condition $2(\dim K+1)<m$, we can describe the inclusion to some extent, which enables us to apply the decomposition of polyhedral product in Example \ref{DJ} to obtain Theorem \ref{second-decomp}.

We abbreviate $\Z_K(\underline{X},\underline{*})$ by $X^K$. We start the proof of Theorem \ref{second-decomp} by showing that the inclusion $X^K\to\Delta_K(X)$ is the fiber inclusion of a homotopy fibration. For this, we apply the following result of Puppe.

\begin{lemma}
[cf. {\cite[Proposition, pp.180]{F}}]
\label{hocolim-fibration}
Let $\{F_i\to E_i\to B\}_{i\in I}$ be an $I$-diagram of homotopy fibrations over a fixed connected base $B$. Then 
$$\underset{I}{\mathrm{hocolim}}\,F_i\to\underset{I}{\mathrm{hocolim}}\,E_i\to B$$ 
is a homotopy fibration.
%Let $E=\bigcup_{\lambda\in\Lambda}E_\lambda$ be a locally finite covering of a CW-complex $E$ by subcomplexes $E_\lambda$. Suppose that for a map $\pi:E\to B$ and any $\sigma\subset\Lambda$ with $E_\sigma\ne\emptyset$, the restriction $\pi\vert_{E_\sigma}:E_\sigma\to B$ is a quasifibration, where $E_\sigma=\bigcap_{\lambda\in\sigma}E_\lambda$. Then so is $\pi:E\to B$.
\end{lemma}

\begin{proposition}
\label{fibration}
If $X$ is connected and $2(\dim K+1)<m$, then there is a homotopy fibration
$$X^K\to\Delta_K(X)\xrightarrow{\pi}X.$$
\end{proposition}

\begin{proof}
Let $\sigma$ be a subset of $[m]$ satisfying $|\sigma|<\frac{m}{2}$. For each point $(x_1,\ldots,x_m)\in\Delta_\sigma(X)$, there is unique $x\in X$ such that more than $\frac{m}{2}$ of $x_i$'s are equal to $x$. Then by assigning such a point $x$ to $(x_1,\ldots,x_m)\in\Delta_\sigma(X)$, we get a map
$$\Delta_\sigma(X)\to X$$
which is identified with the coordinate projection through a homeomorphism $\Delta_\sigma(X)\cong X^{|\sigma|+1}$. Hence this map is a fibration with fiber $(\underline{X},\underline{*})^\sigma$, and yields a diagram of fibrations $\{(\underline{X},\underline{*})^\sigma\to\Delta_\sigma(X)\to X\}_{\sigma\in K}$. So by Lemma \ref{hocolim-fibration} we obtain a homotopy fibration
$$\underset{\sigma\in K}{\mathrm{hocolim}}\,(\underline{X},\underline{*})^\sigma\to\underset{\sigma\in K}{\mathrm{hocolim}}\,\Delta_\sigma(X)\to X.$$ 
For any $\tau\subset\nu\subset[m]$, the inclusions $(\underline{X},\underline{*})^\tau\to(\underline{X},\underline{*})^\nu$ and $\Delta_\tau(X)\to\Delta_\nu(X)$ are cofibrations, implying that there are natural homotopy equivalences

$$\underset{K}{\mathrm{hocolim}}\,(\underline{X},\underline{*})^\sigma\simeq\underset{K}{\mathrm{colim}}\,(\underline{X},\underline{*})^\sigma=X^K\quad\text{and}\quad\underset{K}{\mathrm{hocolim}}\,\Delta_\sigma(X)\simeq\underset{K}{\mathrm{colim}}\,\Delta_\sigma(X)=\Delta_K(X),$$
completing the proof.
%As is noted in the previous section, if $F_1,\ldots,F_t$ are all facets of $K$, then $\Delta_K(X)$ is covered by subcomplexes $\Delta_{F_1}(X),\ldots,\Delta_{F_t}(X)$. By definition, $\Delta_{F_{i_1}}(X)\cap\cdots\cap\Delta_{F_{i_s}}(X)=\Delta_{F_{i_1}\cap\cdots\cap F_{i_s}}(X)$, and the map $\pi\vert_{\Delta_{F_{i_1}}(X)\cap\cdots\cap\Delta_{F_{i_s}}(X)}:\Delta_{F_{i_1}}(X)\cap\cdots\cap\Delta_{F_{i_s}}(X)\to X$ is identified with the projection from the product of $(|F_{i_1}\cap\cdots\cap F_{i_s}|+1)$-copies of $X$. Then the proof is completed by Lemma \ref{union}.
\end{proof}

We next show that the fibration of Proposition \ref{fibration} splits after a suspension. To this end, we use the following.

\begin{lemma}
\label{decomp}
Consider a homotopy fibration $F\xrightarrow{j}E\xrightarrow{\pi}B$ of spaces having the homotopy types of connected CW-complexes. If $\Sigma j:\Sigma F\to\Sigma E$ has a homotopy retraction, then
$$\Sigma E\simeq\Sigma B\vee\Sigma F\vee\Sigma(B\wedge F).$$
\end{lemma}

\begin{remark}
If we assume further that $F$ is of finite type in Lemma \ref{decomp}, it immediately follows from the Leray-Hirsch theorem that the map $\rho$ in the proof of Lemma \ref{decomp} is an isomorphism in cohomology with any field coefficient, implying that $\rho$ is an isomorphism in the integral homology by \cite[Corollary 3A.7]{H}.
\end{remark}

\begin{proof}
Let $r:\Sigma E\to\Sigma F$ be a homotopy retraction of $\Sigma j$, and let $\rho$ be the composite
$$\Sigma E\to\Sigma E\vee\Sigma E\vee\Sigma E\xrightarrow{\Sigma\pi\vee r\vee\Delta}\Sigma B\vee\Sigma F\vee\Sigma(E\wedge E)\xrightarrow{1\vee 1\vee(\pi\wedge r)}\Sigma B^\vee$$
where $A^\vee=A\vee F\vee(A\wedge F)$ for a space $A$. Since $\Sigma E$ and $\Sigma B\vee\Sigma F\vee\Sigma(B\wedge F)$ have the homotopy types of simply connected CW-complexes, it is sufficient to show that $\rho$ is an isomorphism in homology by the J.H.C. Whitehead theorem. We first observe the special case when there is a fiberwise homotopy equivalence $\theta:B\times F\to E$ over $B$. Then it is straightforward to see 
$$\rho_*\circ\theta_*(b\times f)=b\times\hat{\theta}_*(f)+\sum_{|b_i|<|b|}b_i\times f_i$$
for singular chains $b,b_i$ in $B$ and $f,f_i$ in $F$, where we omit writing the suspension isomorphism of homology and $\hat{\theta}$ is a self-homotopy equivalence of $F$ given by the composite
$$\Sigma F\xrightarrow{j}\Sigma(B\times F)\xrightarrow{\theta}\Sigma E\xrightarrow{r}\Sigma F.$$
This readily implies that the map $\rho\circ\theta$ is an isomorphism in homology, and then so is $\rho$. For non-connected $B$, the above is also true if we assume that $r$ is a homotopy retraction of the suspension of the fiber inclusion on each component of $B$. We next consider the general case. Let $B_n$ be the $n$-skeleton of $B$, and let $E_n=\pi^{-1}(B_n)$. We prove that the restriction $\rho\vert_{\Sigma E_n}:\Sigma E_n\to\Sigma B_n^\vee$ is an isomorphism in homology by induction on $n$. Since $B$ is connected, $j$ is homotopic to the composite
$$F\simeq\pi^{-1}(b)\xrightarrow{\text{incl}}E$$
for any $b\in B$. Then $\rho\vert_{\Sigma E_0}:\Sigma E_0\to\Sigma B_0^\vee$ is an isomorphism in homology. Consider the following commutative diagram of homology exact sequences.
\begin{equation}
\label{rho}
\xymatrix{
\cdots\ar[r]&H_*(E_{n-1})\ar[r]\ar[d]^{(\rho\vert_{\Sigma E_{n-1}})_*}&H_*(E_n)\ar[r]\ar[d]^{(\rho\vert_{\Sigma E_n})_*}&H_*(E_n,E_{n-1})\ar[r]\ar[d]^{(\rho\vert_{\Sigma E_n})_*}&\cdots\\
\cdots\ar[r]&H_*(\check{B}_{n-1})\ar[r]&H_*(\check{B}_n)\ar[r]&H_*(\check{B}_n,\check{B}_{n-1})\ar[r]&\cdots}
\end{equation}
By the induction hypothesis, $(\rho\vert_{\Sigma E_{n-1}})_*$ is an isomorphism. Since $B_{n-1}$ is a subcomplex of $B_n$, there is a neighborhood $U\subset B_n$ of $B_{n-1}$ which deforms onto $B_{n-1}$. Then there is a commutative diagram
\begin{equation}
\label{rho-U}
\xymatrix{H_*(E_n,E_{n-1})\ar[r]^(.48)\cong\ar[d]^{(\rho\vert_{\Sigma E_n})_*}&H_*(E_n,\pi^{-1}(U))\ar[d]^{(\rho\vert_{\Sigma E_n})_*}\\
H_*(B_n^\vee,B_{n-1}^\vee)\ar[r]^(.54)\cong&H_*(B_n^\vee,U^\vee)}
\end{equation}
where the basepoint is taken in $B_{n-1}$. On the other hand, by the excision isomorphism, there is a commutative diagram
\begin{equation}
\label{rho-excision}
\xymatrix{H_*(E_n,\pi^{-1}(U))\ar[d]^{(\rho\vert_{\Sigma E_n})_*}&H_*(E_n-E_{n-1},\pi^{-1}(U)-E_{n-1})\ar[d]^{(\rho\vert_{\Sigma(E_n-E_{n-1})})_*}\ar[l]_(.63)\cong\\
H_*(B_n^\vee,U^\vee)&H_*((B_n-B_{n-1})^\vee,(U-B_{n-1})^\vee)\ar[l]^(.63)\cong}
\end{equation}
where the basepoint is taken in $U-B_{n-1}$ and is connected by a path to the formerly chosen basepoint in $B_{n-1}$. Since each connected component of $B_n-B_{n-1}$ is contractible, $E_n-E_{n-1}$ is fiberwise homotopy equivalent to $(B_n-B_{n-1})\times F$ over $B_n-B_{n-1}$, and then so is also $\pi^{-1}(U)-E_{n-1}$ to $(U-B_{n-1})\times F$ over $U-B_{n-1}$. As in the 0-skeleton case, we see that $r$ restricts to a homotopy retraction of the suspension of the fiber inclusion on each component of $B_n-B_{n-1}$. Then by the above trivial fibration case, we obtain that the map $(\rho\vert_{\Sigma(E_n-E_{n-1})})_*$ in \eqref{rho-excision} is an isomorphism. Although the basepoints used in \eqref{rho-U} and \eqref{rho-excision} are distinct, they are connected by a path. In particular, we can juxtapose \eqref{rho-U} and \eqref{rho-excision} to obtain that the right $(\rho\vert_{\Sigma E_n})_*$ in \eqref{rho} is an isomorphism. Thus by the five lemma, the middle $(\rho\vert_{\Sigma E_n})_*$ in \eqref{rho} is an isomorphism. We finally take the colimit to get that the map $\rho$ is an isomorphism in homology as desired, completing the proof.
\end{proof}

To apply Lemma \ref{decomp} to the fibration of Proposition \ref{fibration}, we construct a homotopy retraction of a suspension of the fiber inclusion $j\colon X^K\to\Delta_K(X)$. We first consider a special case.

\begin{proposition}
\label{H-space-trivial}
If $X$ is an H-space having the homotopy type of a CW-complex and $2(\dim K+1)<m$, then the fibration of Proposition \ref{fibration} is trivial.
\end{proposition}

\begin{proof}
Consider the map
$$\varphi\colon X\times X^K\to\Delta_K(X),\quad(x,(x_1,\ldots,x_m))\mapsto(xx_1,\ldots,xx_m).$$
Then $\varphi$ satisfies a homotopy commutative diagram
$$\xymatrix{X^K\ar[r]\ar@{=}[d]&X\times X^K\ar[r]\ar[d]^\varphi&X\ar@{=}[d]\\
X^K\ar[r]&\Delta_K(X)\ar[r]&X}$$
in which two rows are homotopy fibrations. Then it follows from the homotopy exact sequence that $\varphi$ is a weak homotopy equivalence, hence a homotopy equivalence by the J.H.C. Whitehead theorem.
\end{proof}

We set notation. Put $\widehat{X}^K=\bigvee_{\sigma\in K}\widehat{X}^{|\sigma|}$, and let $\epsilon\colon\Sigma X^K\xrightarrow{\simeq}\Sigma\widehat{X}^K$ denote the homotopy equivalence of Example \ref{DJ}. Then a map $f\colon X\to Y$ induces maps $f^K\colon X^K\to Y^K$ and $\hat{f}^K\colon\widehat{X}^K\to\widehat{Y}^K$ which satisfy a commutative diagram
$$\xymatrix{\Sigma X^K\ar[r]^{\epsilon}\ar[d]^{\Sigma f^K}&\Sigma\widehat{X}^K\ar[d]^{\Sigma\hat{f}^K}\\
\Sigma Y^K\ar[r]^{\epsilon}&\Sigma\widehat{Y}^K.}$$

\begin{proposition}
\label{retraction}
If $X$ has the homotopy type of a connected CW-complex and $2(\dim K+1)<m$, then the inclusion $j\colon X^K\to\Delta_K(X)$ has a homotopy retraction after a suspension.
\end{proposition}

\begin{proof}
Let $E:X\to\Omega\Sigma X$ be the suspension map. Since $\Sigma E$ has a retraction, we easily see that the induced map $\Sigma\widehat{E}^K:\Sigma\widehat{X}^K\to\Sigma\widehat{\Omega\Sigma X}{}^K$ has a retraction, say $r$. Then we get a homotopy commutative diagram
$$\xymatrix{\Sigma\widehat{X}^K\ar@{=}[r]\ar@{=}[dd]&\Sigma\widehat{X}^K\ar[r]^{\epsilon^{-1}}\ar[dd]^{\Sigma\widehat{E}^K}&\Sigma X^K\ar[r]^(.45){\Sigma j}\ar[d]^{\Sigma E^K}&\Sigma\Delta_K(X)\ar[d]^{\Sigma\Delta_K(E)}\\
&&\Sigma(\Omega\Sigma X)^K\ar[r]^(.45){\Sigma j}\ar@{=}[d]&\Sigma\Delta_K(\Omega\Sigma X)\ar@{=}[d]\\
\Sigma\widehat{X}^K&\Sigma\widehat{\Omega\Sigma X}{}^K\ar[l]_r&\Sigma(\Omega\Sigma X)^K\ar[l]_\epsilon&\Sigma\Delta_K(\Omega\Sigma X)\ar[l]_{\Sigma r'}}$$
where $\Delta_K(E):\Delta_K(X)\to\Delta_K(\Omega\Sigma X)$ is induced from $E$ and $r'$ is obtained by Proposition \ref{H-space-trivial}. Thus the composite
$$\Sigma\Delta_K(X)\xrightarrow{\Sigma\Delta_K(E)}\Sigma\Delta_K(\Omega\Sigma X)\xrightarrow{\Sigma r'}\Sigma(\Omega\Sigma X)^K\xrightarrow{\epsilon}\Sigma\widehat{\Omega\Sigma X}{}^K\xrightarrow{r}\Sigma\widehat{X}^K\xrightarrow{\epsilon^{-1}}\Sigma X^K$$
is the desired homotopy retraction.
\end{proof}

\begin{proof}
[Proof of Theorem \ref{second-decomp}]
Combine Example \ref{DJ}, Proposition \ref{fibration}, Lemma \ref{decomp} and Proposition \ref{retraction}.
\end{proof}


\begin{thebibliography}{HMR}
\bibitem[ABBCG]{ABBCG}A. Adem, A. Bahri, M. Bendersky, F.R. Cohen, and S. Gitler, {\it On decomposing suspensions of simplicial spaces}, Bol. Soc. Mat. Mexicana (3) {\bf 15} (2009), 91-102.
\bibitem[AC]{AC}A. Adem and F.R. Cohen, {\it Commuting elements and spaces of homomorphisms}, Mathematische Annalen {\bf 338} (2007), 587-626.
\bibitem[ACG]{ACG}A. Adem, F.R. Cohen, and J.M. Gomez, {\it Stable splittings and almost commuting elements in Lie groups}, Math. Proc. Camb. Phil. Soc. {\bf 149} (2010), 455-490.
\bibitem[ACTG]{ACTG}A. Adem, F.R. Cohen, and E. Torres-Giese, {\it Commuting elements, simplicial spaces and filtrations of classifying spaces}, Math. Proc. Camb. Phil. Soc. {\bf 152} (2012), 91-114.
\bibitem[BBCG]{BBCG}A. Bahri, M. Bendersky, F.R. Cohen, and S. Gitler, {\it The polyhedral product functor: a method of decomposition for moment-angle complexes, arrangements and related spaces}, Advances in Math. {\bf 225} (2010), 1634-1668.
\bibitem[F]{F}E.D. Farjoun, {\it Cellular spaces, null spaces and homotopy localization}, Lecture Notes in Mathematics {\bf 1622}, Springer-Verlag, Berlin, 1996.
\bibitem[GM]{GM}M. Goresky and R. MacPherson, {\it Stratified Morse Theory}, Ergebnisse der Math. {\bf 14}, Springer-Verlag, Berlin, Heidelberg, New York, 1988.
\bibitem[GT]{GT}J. Grbi\'c and S. Theriault, {\it The homotopy type of the complement of a coordinate subspace arrangement}, Topology {\bf 46} (2007), 357-396.
\bibitem[H]{H}A. Hatcher, {\it Algebraic Topology}, Cambridge University Press, Cambridge, 2002.
\bibitem[IK1]{IK1}K. Iriye and D. Kishimoto, {\it Decompositions of polyhedral products for shifted complexes}, Advances in Math. {\bf 245} (2013), 716-736.
\bibitem[IK2]{IK2}K. Iriye and D. Kishimoto, {\it Fat wedge filtrations and decomposition of polyhedral products}, \url{arXiv:1412.4866v3}.
\bibitem[J]{J}I.M. James, {\it On H-spaces and their homotopy groups}, Quart. J. Math. Oxford Ser. (2) {\bf 11} (1960), 161-179.
\bibitem[KS]{KS}S. Kallel and I. Saihi, {\it Homotopy groups of diagonal complements}, arXiv:1306.6272.    
\bibitem[Ki]{Ki}S. Kim, {\it Shellable complexes and topology of diagonal arrangements}, Discrete Comput. Geom. {\bf 40} (2008), 190-213.
\bibitem[Ko]{Ko}D.N. Kozlov, {\it A class of hypergraph arrangements with shellable intersection lattice}, J. Comb. Theory, Ser. A {\bf 86} (1999), 169-176.
\bibitem[HMR]{HMR}P. Hilton, G. Mislin, and J. Roitberg, {\it On co-H-spaces}, Comment. Math. Helvetici {\bf 53} (1978), 1-14.
\bibitem[KT]{KT}Y. Kamiyama and S. Tsukuda, {\it The configuration space of the $n$-arms machine in the Euclidean space}, Topology and its Applications {\bf 154} (2007), 1447-1464.
\bibitem[La]{La}F. Labassi, {\it Sur les diagonales \'epaisses et leurs compl\'ementaires}, to appear in J. Homotopy and Related Structures.
\bibitem[Li]{Li}J. Lillig, {\it A union theorem for cofibrations}, Arch. Math. {\bf 24} (1973), 410-415.
\bibitem[M]{M}M.S. Miller, {\it Massey products and k-equal manifolds}, Int. Math. Res. Not. IMRN {\bf 2012}, no. 8, 1805-1821.
\bibitem[MW]{MW}M. Miller and M. Wakefield, {\it Formality of Pascal arrangements}, Proc. Amer. Math. Soc. {\bf 139} (2011), no. 12, 4461-4466.

\bibitem[PRW]{PRW}I. Peeva, V. Reiner, and V. Welker, {\it Cohomology of real diagonal subspace arrangements via resolutions}, Compositio Math. {\bf 117} (1999), 99-115.
\bibitem[ZZ]{ZZ}G.M. Ziegler and R.T. \u{Z}ivaljevi\'c, {\it Homotopy types of subspace arrangements via diagrams of spaces}, Math. Ann. {\bf 295} (1993), 527-548.
\end{thebibliography}
\end{document}